\documentclass[a4paper,12pt,reqno,intlimits]{extarticle}
\usepackage[cp1251]{inputenc}
\usepackage[T2A]{fontenc}
\usepackage[english,russian]{babel}
\usepackage{amsfonts,amsmath,amsxtra,amsthm,amssymb,latexsym,mathrsfs,graphicx,mathtools,cite,nicefrac,url,multirow}
\usepackage{algorithm}
\usepackage{algorithmic}

\DeclareMathOperator*{\argmin}{arg\,min}

\allowdisplaybreaks

\theoremstyle{plain}
\newtheorem{theorem}{Теорема}
\newtheorem{lemma}{Лемма}

\newtheorem{corollary}{Следствие}
\newtheorem{definition}{Определение}

\theoremstyle{definition}
\newtheorem{remark}{Замечание}
\newtheorem{example}{Пример}

\numberwithin{theorem}{section}
\numberwithin{lemma}{section}
\numberwithin{proposition}{section}
\numberwithin{corollary}{section}
\numberwithin{remark}{section}
\numberwithin{definition}{section}
\numberwithin{example}{section}

\DeclarePairedDelimiter\bracket{(}{)}
\newcommand{\br}[1]{\bracket*{#1}}
\DeclarePairedDelimiter\figbracket{\{}{\}}
\newcommand{\fbr}[1]{\figbracket*{#1}}

\DeclarePairedDelimiter\angbracket{\langle}{\rangle}
\newcommand{\abr}[1]{\angbracket*{#1}}
\DeclarePairedDelimiter\normbracket{\|}{\|}
\newcommand{\nbr}[1]{\normbracket*{#1}}

\title{Адаптивные градиентные методы для некоторых классов задач негладкой оптимизации}
\date{}

\begin{document}
\author{Ф.С. Стонякин}
\maketitle

\bigskip

\section*{Введение}

Во многих прикладных задачах возникает необходимость подходящих алгоритмических методов для задач негладкой выпуклой оптимизации. Однако оценки эффективности таких процедур в случае большой размерности переменных весьма пессимистичны. Так, к примеру, $\varepsilon$-точное решение по функции задачи выпуклой негладкой оптимизации возможно достичь за $O(\varepsilon^{-2})$ обращений к подпрограмме нахождения (суб)градиента и в общем случае такая оценка не улучшаема \cite{NemYud_1979}. Для гладких задач оценки эффективности выше, что приводит к естественной идее для негладких задач обосновать возможность использования какого-нибудь приближения оптимизационной модели к гладкому случаю. Эту идею, в частности, реализуют так называемые {\it универсальные методы}, исследованию которых было положено начало в работе \cite{Nesterov_2015}. Универсальные градиентные методы основаны на построении для задач выпуклой оптимизации с гёльдеровым (суб)градиентом целевого функционала аналога стандартной квадратичной интерполяции с искусственно введённой погрешностью. Универсальность метода при этом понимается как возможность адаптивной настройки при работе метода на оптимальный в некотором смысле уровень гладкости задачи и величину, соответствующую константе Гёльдера $L_{\nu}$ (суб)градиента целевого функционала. Оказывается, что возможность такой настройки может позволить экспериментально для некоторых задач улучшить скорость сходимости по сравнению с оптимальными теоретическими оценками \cite{Nesterov_2015}.

Искусственную неточность для негладких задач можно вводить по-разному \cite{Ston_new, StonyakinTrudy}. При этом естественно возникает проблема описания влияния погрешностей задания целевого функционала и градиента на оценки скорости сходимости методов. Для градиентных методов выпуклой оптимизации известен подход, основанный на недавно предложенной концепции неточного оракула \cite{s1, DevolderThesis}. Известно, что для неускоренных градиентных методов в оценках не происходит накопления величин, связанных с погрешностями. Однако для оптимальных при отсутствии погрешностей на классе гладких задач ускоренных методов в итоговой оценке скорости сходимости величины погрешностей могут накапливаться. Концепции неточного оракула были обобщены в \cite{s2, inexact_model_2019}, где были введены понятия ($\delta, L$)-модели и ($\delta, L, \mu$)-модели целевой функции для задач оптимизации. Суть данных обобщений в том, что функция $\langle \nabla f(x), y - x \rangle$ в оптимизационной модели заменяется на некоторую абстрактную выпуклую по первой переменной функцию $\psi(y, x)$. Это позволяет расширить класс задач, к которым применимы указанные подходы \cite{inexact_model_2019}.

Различным методам градиентного типа посвящены всё новые современные работы \cite{s2, Bauschke_2017, Necoara_2019, inexact_model_2019, s3}. В частности, недавно в \cite{Bauschke_2017} введены условия относительной гладкости оптимизируемого функционала, предполагающие замену липшицевости градиента на ослабленный вариант
\begin{equation}\label{eq01new}
f(y)\leqslant f(x)+\abr{\nabla f(x),y-x}+LV(y, x),
\end{equation}
где $V(y, x)$~--- широко используемый в оптимизации аналог расстояния между точками $x$ и $y$, который называют \emph{дивергенцией Брэгмана}. Обычно \emph{дивергенция Брэгмана} вводится на базе вспомогательной 1-сильно выпуклой функции $d$ (порождает расстояния), которая дифференцируема во всех точках выпуклого замкнутого множества $Q$:
\begin{equation}\label{3}
V(y,x)=d(y)-d(x)-\langle \nabla d(x),y-x\rangle \quad \forall x,y\in Q,
\end{equation}
где $\langle \cdot,\cdot\rangle$~--- скалярное произведение в $\mathbb{R}^n$. В частности, для стандартной евклидовой нормы $\|\cdot\|_2$ и расстояния в $\mathbb{R}^n$ можно считать, что $V(y,x) = d(y - x) = \frac{1}{2} \|y - x\|_2^2$ для произвольных $ x, y \in Q$. Однако часто возникает необходимость использовать и неевклидовы нормы. Более того, рассмотренное в \cite{Bauschke_2017, s3} условие относительной гладкости предполагает лишь выпуклость (но не сильную выпуклость) порождающей функции $d$. Как показано в \cite{s3}, концепция относительной гладкости позволяет применить вариант градиентного метода для некоторых задач, которые ранее решались лишь с помощью методов внутренней точки. В частности, речь идет об известной задаче построения оптимального эллипсоида, покрывающего заданный набор точек. Эта задача, в частности, представляет интерес для статистики и анализа данных. Отметим в этой связи также предложенный недавно в \cite{Lu_1710.04718v3} подход к задачам негладкой оптимизации, связанный с релаксацией условия Липшица, которая предполагает замену ограниченности нормы субградиента $\|\nabla f(x)\|_* \leqslant M_f$ так называемой {\it относительной липшицевостью}
$$
\|\nabla f(x)\|_* \leqslant \frac{M_f \sqrt{2V(y,x)}}{\|y-x\|} \quad \forall x, y \in Q, \; y \neq x.
$$
При этом порождающая функция $d$ не обязательно сильно выпукла. В работе \cite{Lu_1710.04718v3} предложены детерминированный и стохастический алгоритмы зеркального спуска для задач оптимизации выпуклого относительно липшицева целевого функционала. Отметим также, что в \cite{Stonyakin_1908.00218v4} показана оптимальность зеркального спуска для задач математического программирования с квазивыпуклыми дифференцируемыми гёльдеровыми целевыми функционалами и выпуклыми липшицевыми функционалами ограничений.

Предлагаемая статья посвящена в основном развитию упомянутых вышей идей. В первом разделе предложена модификация концепции ($\delta, L$)-модели целевой функции с переменной неточностью, которая позволяет учесть возможность неточного задания не только значения целевой функции, но и самой модели. В частности, для стандартной модели $\psi(y, x) = \langle \nabla f(x), y - x \rangle$ описывается ситуация погрешности задания целевого функционала $f$, а также погрешности задания (суб)градиента $\nabla f$.

В первых трёх разделах работы мы выделяем (в модельной общности по аналогии с \cite{s2}) класс задач, для которых верно следующее неравенства c параметрами $\delta_1, \delta_2, \gamma, \Delta \geqslant 0$:
\begin{equation}
\label{eqconceptnew}
\begin{split}
f(x)+\abr{\widetilde{\nabla}f(x),y-x}-\delta_2-\gamma\nbr{y-x}\leqslant f(y) \leqslant\\
\leqslant f(x)+\abr{\widetilde{\nabla}f(x),y-x}+\frac{L}{2}\nbr{y-x}^2+\Delta\nbr{y-x}+\delta_1\,\,\,\forall x,y\in Q.
\end{split}
\end{equation}

Смысл такого обобщения заключается в том, что возможны различные значения параметров $\gamma$ и $\Delta$ в \eqref{eqconceptnew} и влияние величин $\delta_1$ и $\Delta$ на итоговое качество решения может быть уменьшено. Случай $\gamma > 0$ более подробно рассмотрен подробно в \cite{StonyakinTrudy}, а в настоящей работе мы полагаем $\gamma = 0$. Это предположение вполне естественно, если рассматривать $\Delta$ как искусственную неточность для негладкой задачи. В этом случае $\Delta$
в неравенстве
\begin{equation}\label{NonSmooth}
f(x)+\abr{\widetilde{\nabla}f(x),y-x} \leqslant f(y)\leqslant f(x)+\abr{\widetilde{\nabla}f(x),y-x}+\frac{L}{2}\nbr{y-x}^2+\Delta\nbr{y-x}
\end{equation}
вполне можно рассматривать как оценку скачков субдифференциалов $f$ вдоль всевозможных векторных отрезков $[x; y]$ \cite{Ston_new}.
Несложно понять, что если $f$ в \eqref{NonSmooth} удовлетворяет условию Липшица с константой $M>0$, то $\Delta \leqslant 2M$. Отметим, что возможна ситуация, когда $\Delta$ существенно меньше $M$ \cite{Ston_new}. Похожие на \eqref{NonSmooth} условия рассматривались в \cite{DevolderThesis,Lan_2016} для случая, когда $f$ представим в виде суммы гладкого и негладкого целевого функционала. В настоящей работе рассмотрен более широкий класс целевых функционалов, которые не обязательно представимы в виде суммы гладкого и негладкого слагаемых. В итоге предложена общая концепция неточной модели целевой функции, которая могла бы описать все указанные выше ситуации. Далее разделе 4 введена аналогичная концепция неточной модели для оператора поля вариационного неравенства, а также для седловой задачи. Выписаны оценки скорости сходимости соответствующего адаптивного варианта проксимального зеркального метода. В последнем разделе статьи рассмотрены аналоги субградиентных схем с переключениями для задач выпуклой оптимизации с ограничениями. При этом рассмотрены предположения, близкие к недавно предложенному условию относительной липшицевости \cite{Lu_1710.04718v3}, что позволило получить оценку качества решения с относительной точностью для задачи минимизации однородного выпуклого функционала при достаточно общих предположениях.

Работа состоит из введения, пяти основных разделов и заключения.

В разделе 1 обобщено ранее предложенное понятие $(\delta,L)$-модели целевой функции \cite{s2} в запрошенной точке и введена концепция ($\delta, \Delta, L$)-модели функции (определение \ref{def1}). Введена также более общая концепция ($\delta, \Delta, L, \mu$)-модели целевой функции, которую можно в частности применить на классе гладких сильно выпуклых функционалов. Предложен аналог градиентного метода (алгоритм \ref{alg3}) градиентного типа с адаптивной настройкой параметров неточности этой модели.

В разделе 2 предложен вариант быстрого градиентного метода (алгоритм \ref{alg2}) для задач выпуклой минимизации с адаптивным выбором шага и адаптивной настройкой на величины параметров ($\delta, \Delta, L$)-модели и получена оценка качества найденного решения.

Раздел 3 посвящен специальному классу задач выпуклой негладкой оптимизации, к которым применима концепция определения \ref{def1} ($\delta = 0$, $\Delta>0$). Показано, что для таких задач возможно модифицировать алгоритм \ref{alg2} так, чтобы гарантированно достигалось $\varepsilon$-точное решение задачи минимизации $f$ за
$$O\br{\sqrt{\frac{L}{\varepsilon}}}+O\br{\frac{\Delta^2}{\varepsilon^2}}$$
обращений к (суб)градиенту целевого функционала. По сути, получен некоторый аналог результата (\cite{DevolderThesis}, п. 4.7.2), однако оптимизируемый функционал уже не обязательно имеет вид суммы гладкого и негладкого слагаемого и рассмотрена модельная общность. Отметим, что при этом для адаптивного ускоренного метода оценка может увеличиться на некоторый логарифмический множитель вида $O(log_2 (\varepsilon^{-3}))$, для адаптивного неускоренного~--- на множитель вида $O(log_2 (\varepsilon^{-1}))$. Отметим также, что параметр $\Delta$ можно понимать и как величину, которая соответствует неточно заданному (суб)градиенту.

Далее в разделе 4 показано, как можно вводить аналог концепции ($\delta, \Delta, L$)-модели оптимизируемого функционала, но уже для вариационных неравенств и седловых задач.

Наконец, в последнем разделе 5 обсуждаются альтернативные подходы к алгоритмическим схемам для задач негладкой оптимизации с липшицевыми выпуклыми функциональными ограничениями. Предложен адаптивный метод с заменой обычного условия Липшица на некоторые его аналоги. Рассмотрены условия, близкие к относительной липшицевости \cite{Lu_1710.04718v3}, которая в частности не предполагает сильной выпуклости дивергенции Брэгмана. Это позволило предложить метод, который гарантирует достижение качества приближённого решения с фиксированной относительной точностью относительной для задачи минимизации выпуклого однородного целевого функционала $f$ в случае, когда субдифференциал $f$ в нулевой точке может не содержать 0 как внутреннюю точку.

Всюду далее через $\|\cdot\|$ обозначается норма в пространстве $\mathbb{R}^n$, а через $\|\cdot\|_*$~--- норма в двойственном пространстве.

\section{Концепция $(\delta,\Delta,L)$-модели функции в запрошенной точке и оценка скорости сходимости для градиентного метода}

Введем анонсированный выше аналог понятия $(\delta,\Delta,L)$-модели целевой функции, который учитывает погрешность $\Delta$ задания градиента и применим также для задач с относительно гладкими целевыми функционалами \cite{s3}.

\begin{definition}\label{def1}
Будем говорить, что $f$ допускает $(\delta,\Delta,L)$-модель $\psi$ в точке $x \in Q$, если для некоторой выпуклой по первой переменной функции $\psi(y,x)$ такой, что при всяком $x$ $\psi(x,x)=0$, будет верно неравенство
\begin{equation}\label{eq30a}
f(x)+\psi(y,x)-\delta_1 \leqslant f_{\delta}(x) + \psi(y,x) \leqslant f(y)\leqslant f_{\delta}(x)+\psi(y,x)+\delta_2+\Delta\nbr{y-x}+LV(y,x)
\end{equation}
для произвольных $x,y\in Q$ при фиксированных $\delta, \delta_1, \delta_2 \geqslant 0$ и $f_{\delta}(x) \in [f(x) - \delta, f(x)]$ $\forall x \in Q$.
\end{definition}

Всюду далее будем полагать, что $\delta = \delta_1 = \delta_2$, а также $f_{\delta}(x) \in [f(x) - \delta, f(x)]$. Отметим, что предложенная концепция есть модификация введенного в \cite{StonyakinTrudy} понятия $(\delta,\Delta, \gamma, L)$-модели целевого функционала в запрашиваемой точке.

\begin{remark}
Для неускоренного метода возможно заменить левое неравенство в \eqref{eq30a} на ослабленный вариант
\begin{equation}\label{eq30b}
f(x_*)\geqslant f(x)+\psi(x_*,x)-\delta,
\end{equation}
где $x_*$~--- ближайшее к $x$ решение задачи минимизации $f$ c точки зрения дивергенции Брэгмана $V(x_*, x)$. Неравенство \eqref{eq30b} будет, в частности, верно для задачи минимизации квазивыпуклой целевой функции при достаточно малой величине погрешности градиента.
\end{remark}

Можно также ввести и аналог концепции ($\delta, L, \mu$)-оракула \cite{DevolderThesis} в оптимизации с переменной неточностью. При $\mu>0$ эта концепция позволяет обосновать скорость сходимости предлагаемого нами метода, близкую к линейной.

\begin{definition}\label{def02}
Будем полагать, что $f$ допускает $(\delta,\Delta, L,\mu)$-модель $\psi(y,x)$, если для произвольных $x, y \in Q$ верно
\begin{equation}\label{eq311a}
f(y)\leqslant f_{\delta}(x)+\psi(y,x)+ \Delta\|y - x\| + \delta + LV(y,x),
\end{equation}
а также
\begin{equation}\label{eq311b}
f(x) - \delta +\psi(x_*,x)+\mu V(x_*,x)\leqslant f_{\delta}(x) +\psi(x_*,x)+\mu V(x_*,x) \leqslant f(x_*),
\end{equation}
где $x_*$~--- ближайшее к $x$ решение задачи минимизации $f$ c точки зрения дивергенции Брэгмана $V(x_*, x)$. Неравенство \eqref{eq311b} будет, в частности, верно для задачи минимизации сильно квазивыпуклой целевой функции \cite{Necoara_2019} при достаточно малой величине погрешности градиента.
\end{definition}

\begin{example}\label{ex2}
В качестве примера по аналогии с примером 1 из \cite{StonyakinTrudy} отметим задачу сильно выпуклой композитной оптимизации $f(x)=g(x)+h(x)\rightarrow\min$, где $g$~--- гладкая выпуклая функция, а $h$~--- выпуклая не обязательно гладкая функция простой структуры. Если при этом для градиента $\nabla g$ задано его приближение $\widetilde{\nabla}g$:
$\nbr{\widetilde{\nabla}g(x)-\nabla g(x)}\leqslant\Delta,$ то можно положить
$$\psi(y,x) = \langle \widetilde{\nabla}g(x),y-x \rangle+h(y)-h(x)$$
и в случае $\mu$-сильной выпуклости $g$ или $h$ будет верно \eqref{eq311a}.
\end{example}

Для задачи минимизации функционала, допускающего $(\delta, \Delta, L, \mu)$-модель в произвольной запрошенной точке предложим такой метод.

\begin{algorithm}
\caption{Адаптивный градиентный метод для функций, допускающих $(\delta, \Delta, L, \mu)$-модель в запрошенной точке.}
\label{alg3}
\begin{algorithmic}[1]
\REQUIRE $x^0\in Q$~--- начальная точка, $V(x_*,x^0)\leqslant R^2$, параметры $L_0, \Delta_0, \delta_0>0:\; 2\mu<L_0\leqslant2L,\; \Delta_0\leqslant2\Delta,\;\delta_0\leqslant2\delta$, $N$ --- фиксированное кол-во шагов.
\STATE $L_{k+1}:=\max\left\{\mu, \nicefrac{L_k}{2}\right\},\;\Delta_{k+1}:=\nicefrac{\Delta_k}{2},\;\delta_{k+1}:=\nicefrac{\delta_k}{2}$.
\STATE $x^{k+1}:=\text{arg}\min\limits_{x\in Q} \{ \psi(x,x^k)+L_{k+1}V(x,x^k)\}$.
\REPEAT
\IF{$f_{\delta}(x^{k+1})\leqslant f_{\delta}(x^k)+\psi(x^{k+1},x^k)+L_{k+1}V(x^{k+1},x^k)+\delta_{k+1}+\Delta_{k+1}\nbr{x^{k+1}-x^k}$}
\STATE $k:=k+1$ и выполнение п.~1.
\ELSE
\STATE $L_{k+1}:=2\cdot L_{k+1}\;\delta_{k+1}:=2\cdot\delta_{k+1};\;\Delta_{k+1}:=2\cdot\Delta_{k+1}$ и выполнение п.~2.
\ENDIF
\UNTIL{$k \geqslant N$}
\ENSURE $y^{k+1}:=\text{arg}\min\limits_{i=0,\ldots,k}f(x^{i+1})$.
\end{algorithmic}
\end{algorithm}

\newpage

Всюду далее условимся полагать, что $\prod\limits_{i=k+1}^k a_i = 1$ для некоторой числовой последовательности $a_i$. Справедлива следующая

\begin{theorem}\label{th03}
Пусть $f$ имеет $(\delta,\Delta,L,\mu)$-модель в каждой точке $x\in Q$. Тогда после $k$ итераций справедливо неравенство
$$f(y^{k+1})-f(x_*)\leqslant\frac{1}{\sum\limits_{i=0}^k\frac{1}{L_{i+1}}\prod\limits_{j=i+1}^k\br{1-\frac{\mu}{L_j}}}\cdot$$
$$\cdot\br{\prod\limits_{i=0}^k\br{1-\frac{\mu}{L_{i+1}}}V(x_*,x^0)+\sum\limits_{i=0}^k\frac{\delta+\delta_{i+1}+\Delta_{i+1}\nbr{x^{i+1}-x^i}}{L_{i+1}}\prod\limits_{j=i+1}^k\br{1-\frac{\mu}{L_j}}}.$$

Отметим, что вспомогательная задача п. 2 листинга алгоритма \ref{alg3} решается не более
\begin{equation}\label{Oracle_Estim_Mu}
2k + \max\left\{\log_2\frac{2L}{L_0}, \log_2\frac{2\delta}{\delta_0}, \log_2\frac{2\Delta}{\Delta_0}\right\}
\end{equation}
раз.
\end{theorem}
\begin{proof}
Введем обозначения: $\widehat{\delta}_{k+1}:=\delta_{k+1}+\Delta_{k+1}\nbr{x^{k+1}-x^k}$. После $k$ итераций алгоритма \ref{alg3} получаем
$$ 0 \leqslant \psi(x,x^k) - \psi(x^{k+1},x^k) + L_{k+1}V(x,x^k)-L_{k+1}V(x,x^{k+1})-L_{k+1}V(x^{k+1},x^k),\text{ откуда}$$
\begin{equation}
\label{eq3}
L_{k+1}V(x, x^{k+1})\leqslant \psi(x,x^k)-\psi(x^{k+1}, x^k) + L_{k+1}V(x,x^k)-L_{k+1}V(x^{k+1},x^k).
\end{equation}
Согласно неравенству \eqref{eq311a}:
$$-L_{k+1}V(x^{k+1}, x^k) \leqslant \widehat{\delta}_{k+1} - f_{\delta}(x^{k+1}) + f_{\delta}(x^k) + \psi(x^{k+1}, x^k) .$$
Применяя теперь $\eqref{eq3}$, получаем
\begin{equation}\label{eq4}
L_{k+1}V(x, x^{k+1}) \leqslant \widehat{\delta}_{k+1} - f(x^{k+1}) + f_{\delta}(x^k) + \psi(x, x^k) + L_{k+1}V(x,x^k)+\delta.
\end{equation}
Пусть $x=x_*$. Тогда, учитывая \eqref{eq311b}, имеем
$f_{\delta}(x^k) + \psi(x_*,x^k) \leqslant f(x_*) - \mu V(x_*, x^k).$

Применим это неравенство к \eqref{eq4}: $L_{k+1}V(x_*, x^{k+1}) \leqslant \delta+\widehat{\delta}_{k+1} + f(x_*) - f(x^{k+1})+(L_{k+1}-\mu)V(x_*,x^k)$. Далее,
$$V(x_*,x^{k+1})\leqslant\frac{f(x_*)-f(x^{k+1})}{L_{k+1}}+\frac{\delta+\widehat{\delta}_{k+1}}{L_{k+1}}+\br{1-\frac{\mu}{L_{k+1}}}V(x_*,x^k)\leqslant$$
$$\leqslant\frac{f(x_*)-f(x^{k+1})}{L_{k+1}}+\frac{1}{L_k}\br{1-\frac{\mu}{L_{k+1}}}\br{f(x_*)-f(x^k)}+\frac{\delta+\widehat{\delta}_{k+1}}{L_{k+1}}+$$
$$+\frac{\delta+\widehat{\delta}_k}{L_k}\br{1-\frac{\mu}{L_{k+1}}}+\br{1-\frac{\mu}{L_{k+1}}}\br{1-\frac{\mu}{L_k}}V(x_*,x^{k-1})\leqslant$$
$$\leqslant\sum\limits_{i=0}^k\frac{f(x_*)-f(x^{i+1})}{L_{i+1}}\prod\limits_{j=i+1}^k\br{1-\frac{\mu}{L_j}}+$$
$$+\sum\limits_{i=0}^k\frac{\delta+\widehat{\delta}_{i+1}}{L_{i+1}}\prod\limits_{j=i+1}^k\br{1-\frac{\mu}{L_j}}+\prod\limits_{i=0}^k\br{1-\frac{\mu}{L_{i+1}}}V(x_*,x_0).$$
С учетом $y^{k+1}=\argmin\limits_{i=\overline{0,k}}f(x^{i+1})$ и $V(x_*,x^{k+1})\geqslant0$ имеем $f(y^{k+1})-f(x_*)\leqslant$
$$\leqslant\frac{1}{\sum\limits_{i=0}^k\frac{1}{L_{i+1}}\prod\limits_{j=i+1}^k\br{1-\frac{\mu}{L_j}}}\br{\prod\limits_{i=0}^k\br{1-\frac{\mu}{L_{i+1}}}V(x_*,x^0)+\sum\limits_{i=0}^k\frac{\delta+\widehat{\delta}_{i+1}}{L_{i+1}}\prod\limits_{j=i+1}^k\br{1-\frac{\mu}{L_j}}}.$$

Оценка \eqref{Oracle_Estim_Mu} обосновывается аналогично п. 2 доказательства теоремы 2.1 из \cite{StonyakinTrudy}.
\end{proof}

\begin{remark}
Если убрать соответствующее ограничение пункта 1 листинга алгоритма \ref{alg3} и допустить возможность при некотором $k \geqslant 0$ $L_{k+1}<\mu$, то в таком случае на этой итерации будет верно $f(x^{k+1})-f(x_*)\leqslant\delta+\widehat{\delta}_{k+1}$.
\end{remark}

\begin{remark}\label{RemConst}
Оценка \eqref{Oracle_Estim_Mu} показывает, что в среднем трудоемкость итерации предложенного адаптивного алгоритма превышает трудоемкость неадаптивного метода не более, чем в постоянное число раз. Отметим также, что при $k = 0,1,2,...$ $L_{k+1} \leqslant 2CL$, $C = \max\left\{1, \,\frac{2\delta}{\delta_0},\,\frac{2\Delta}{\Delta_0}\right\}$.
\end{remark}

\begin{corollary}
При $\mu=0$ полученная оценка качества решения принимает вид:
\begin{equation}\label{eq44}
f(y^{k+1})-f(x_*)\leqslant\frac{V(x_*,x^0)}{\sum\limits_{i=0}^k\frac{1}{L_{i+1}}}+\left(\sum\limits_{i=0}^k\frac{1}{L_{i+1}}\right)^{-1}\sum\limits_{i=0}^k\frac{\delta_{i+1}+\Delta_{i+1}\nbr{x^{i+1}-x^i}}{L_{i+1}} + \delta \leqslant
\end{equation}
$$
\leqslant 2CLV(x_*,x^0)+\left(\sum\limits_{i=0}^k\frac{1}{L_{i+1}}\right)^{-1}\sum\limits_{i=0}^k\frac{\delta_{i+1}+\Delta_{i+1}\nbr{x^{i+1}-x^i}}{L_{i+1}} + \delta.
$$
\end{corollary}

\section{Оценка скорости сходимости для варианта быстрого градиентного метода, использующего концепцию $(\delta, \Delta, L)$-модели целевой функции в запрошенной точке}

Оказывается, что концепция ($\delta, \Delta, L$)-модели оптимизируемого функционала позволяет получить аналог теоремы \ref{th03} для адаптивного варианта быстрого градиентного метода Ю.\,Е.\,Нестерова (БГМ) \cite{paper:Nesterov1983}. Смысл использования этого метода в том, что при отсутствии погрешностей он даёт оптимальные оценки скорости сходимости на классе задач выпуклой оптимизации с липшицевым градиентом. Мы отправляемся от \cite{s2}, где предложен адаптивный быстрый градиентный метод с оракулом, использующий $(\delta, L)$-модель целевой функции в запрошенной точке (в нашей терминологии это случай $\Delta = 0$). Здесь уже существенна 1-сильная выпуклость прокс-функции $d(x)$ относительно нормы. Для варианта быстрого градиентного метода будем предполагать, что существует $(\delta, \Delta, L)$-модель для $f(x)$ в любой точке $x \in Q$, причём неравенство \eqref{eq30a}
заменяется на
$$
f(x)+\psi(y,x)-\delta_1 \leqslant f_{\delta}(x) + \psi(y,x) \leqslant f(y)\leqslant f_{\delta}(x)+\psi(y,x)+\delta_2+\Delta\nbr{y-x}+\frac{L}{2}\|y - x\|^2.
$$

\begin{algorithm}[ht]
\caption{Быстрый градиентный метод с оракулом, использующий $(\delta, \Delta, L)$-модель в запрошенной точке.}
\label{alg2}
\begin{algorithmic}[1]
\REQUIRE $x^0\in Q$~--- начальная точка, $V(x_*,x^0)\leqslant R^2$, параметры $L_0>0,\;\Delta_0>0,\;\delta_0>0$\\
($L_0\leqslant2L,\;\Delta_0\leqslant2\Delta,\;\delta_0\leqslant2\delta$).
\STATE \textbf{0-шаг:}
$
y^0 := x^0,\,
u^0 := x^0,\,
L_1 := \frac{L_0}{2},\,
\Delta_1 := \frac{\Delta_0}{2},\,
\delta_1 := \frac{\delta_0}{2},\,
\alpha_0 := 0,\,
A_0 := \alpha_0
$
\FOR{$k = 1, ...$}
\STATE Находим наибольший корень $\alpha_{k+1} : A_k + \alpha_{k+1} = L_{k+1}\alpha^2_{k+1}$;
$$A_{k+1} := A_k + \alpha_{k+1},\;
y_{k+1} := \frac{\alpha_{k+1}u^k + A_k x^k}{A_{k+1}},\;
\phi_{k+1}(x) = V(x, u^k) + \alpha_{k+1}\psi(x, y^{k+1});$$
$$u^{k+1} := \argmin\limits_{x \in Q}\phi_{k+1}(x),\;
x^{k+1} := \frac{\alpha_{k+1}u^{k+1} + A_k x^k}{A_{k+1}};$$
\IF{$f_{\delta}(x^{k+1}) \leqslant f_{\delta}(y^{k+1}) + \psi(x^{k+1}, y^{k+1})+ \frac{L_{k+1}}{2}\| x^{k+1} - y^{k+1}\|^2 + \Delta_{k+1}\|x^{k+1} - y^{k+1}\| + \delta_{k+1}$}
\STATE $L_{k+2} := \frac{L_{k+1}}{2}$, $\Delta_{k+2} := \frac{\Delta_{k+1}}{2}$ и $\delta_{k+2} := \frac{\delta_{k+1}}{2}$, потом перейти к следующему шагу
\ELSE
\STATE $L_{k+1} := 2L_{k+1}$, $\Delta_{k+1} := 2\Delta_{k+1}$ и $\delta_{k+1} := 2\delta_{k+1}$ и повторить текущий шаг
\ENDIF
\ENDFOR
\end{algorithmic}
\end{algorithm}

\begin{remark}
\leavevmode
\label{remark_maxmin}
Аналогично замечанию \ref{RemConst} для всякого $k \geqslant 0$ при некотором $C \geqslant 1$ выполнено
$L_{k} \leqslant 2CL$. При $k = 0$ это верно из того, что $L_0 \leqslant 2L$. Для $k \geqslant 1$ это следует из того, что мы выйдем из внутреннего цикла, где подбирается $L_k$, ранее, чем $L_{k}$ станет больше $2CL$. Выход из цикла гарантируется тем, что согласно предположению существует $(\delta, \Delta, L)$-модель для $f(x)$ в любой точке $x \in Q$.
\end{remark}

По схеме рассуждений \cite{s2} можно проверить следующий результат.

\begin{theorem}
\label{mainTheoremDL}
Пусть $V(x_*, x^0) \leqslant R^2$, где $x^0$~--- начальная точка, а $x_*$~--- ближайшая точка минимума к точке $x^0$ в смысле дивергенции Брэгмана. Для предложенного алгоритма 2 после $k$ итераций выполнено следующее неравенство:
\begin{equation*}
f(x^k) - f(x_*) \leqslant \frac{R^2}{A_k} + \frac{1}{A_k}\sum\limits_{i = 0}^{k-1}(\Delta_{i+1}\nbr{x^{i+1} - y^{i+1}} + \delta_{i+1} + \delta)A_{i+1}.
\end{equation*}
Отметим, что вспомогательная задача пункта 3 листинга алгоритма \ref{alg2} решается не более
\begin{equation}\label{Oracle_Estim FGM}
2k + \max\left\{\log_2\frac{2L}{L_0}, \log_2\frac{2\delta}{\delta_0}, \log_2\frac{2\Delta}{\Delta_0}\right\}
\end{equation}
раз.
\end{theorem}

Полученную оценку можно несколько конкретизировать с использованием следующего вспомогательного утверждения,
которое также проверяется по схеме рассуждений \cite{s2} (по сути оно уже доказано в \cite{s2} при $C=1$).

\begin{lemma}
\label{lemma_maxmin_1}
Пусть для последовательности $\alpha_k$ выполнено
$$
\alpha_0 = 0,\,\,\,
A_k = \sum\limits_{i = 0}^{k}\alpha_i,\,\,\,
A_k = L_{k}\alpha_k^2,\,\,\,
$$
где для фиксированного $C \geqslant 1$ верно $L_k \leqslant 2CL$ при всяком $k\geq0$ согласно замечанию \ref{remark_maxmin} выше.
Тогда для любого $k \geqslant 1$ верно следующее неравенство
\begin{align}
\label{lemma_maxmin_1_1}
A_k \geqslant \frac{(k+1)^2}{8CL}.
\end{align}
\end{lemma}

Поэтому из теоремы \ref{mainTheoremDL} вытекает

\begin{corollary}
\label{CorollaryDL}
Пусть $V(x_*, x^0) \leqslant R^2$, где $x^0$~--- начальная точка, а $x_*$~--- ближайшая точка минимума к точке $x^0$ в смысле дивергенции Брэгмана. Для предложенного алгоритма 2 выполнено следующее неравенство:
\begin{equation*}
f(x^k) - f(x_*) \leqslant \frac{8CLR^2}{(k+1)^2} + \frac{1}{A_k}\sum\limits_{i = 0}^{k-1}(\Delta_{i+1}\nbr{x^{i+1} - y^{i+1}} + \delta_{i+1} + \delta)A_{i+1}.
\end{equation*}
\end{corollary}

\section{О применимости предложенной концепции неточной модели функции к негладким задачам за счет введения искусственных неточностей}

На величину $\Delta$ в \eqref{eq44} можно смотреть как на характеристику негладкости функционала $f$. Точнее говоря, $\Delta$ можно понимать, например, как верхнюю оценку суммы диаметров субдифференциалов $f$ в точках негладкости вдоль всевозможных векторных отрезков $[x; y]$ из области определения $f$. Оказывается, в случае известной величины $\Delta< + \infty$ возможно несколько модифицировать алгоритм \ref{alg3}, обеспечив уменьшение $\Delta_{i+1}\|x^{i+1}-x^i\|$ в \eqref{eq44} до любой заданной величины. Это позволит показать оптимальность данного метода в теории нижних оракульных оценок \cite{NemYud_1979} с точностью до логарифмического множителя.

Покажем, как это возможно сделать в предположении, что целевой функционал $f$ допускает ($\delta, \Delta, L$)-модель $\psi$. Заметим, что при этом в настоящем пункте мы всюду требуем $1$-сильную выпуклость прокс-функции в определении \ref{def1}. Будем обозначать через $f^* = f(x_*)$ искомое минимальное значение целевой функции.

Пусть на (k+1)-й итерации алгоритма \ref{alg3} ($k=0,1,\ldots,N-1$) верно неравенство $L\leqslant L_{k+1}\leqslant 2L$ (как показано в п.~2 доказательства теоремы~2.1 из \cite{StonyakinTrudy}, этого можно всегда добиться выполнением не более чем постоянного числа операций п.~2 листинга алгоритма \ref{alg3}). Для каждой итерации алгоритма \ref{alg3} ($k=0,1,\ldots,N-1$) предложим такую процедуру:
\begin{equation}\label{eq12}
\fbox{\begin{minipage}{25em}
Повторяем операции п.~2 $p$ раз, увеличивая $L_{k+1}$ в два раза при неизменной $\Delta_{k+1}\leqslant2\Delta$.
\end{minipage}}
\end{equation}

Процедуру \eqref{eq12} остановим в случае выполнения одного из неравенств:
\begin{equation}\label{eq13}
\Delta_{k+1}\nbr{x^{k+1}-x^k}\leqslant\frac{\varepsilon}{2}
\end{equation}
или
\begin{equation}\label{eq14}
f(x^{k+1})\leqslant f(x^k)+ \psi(x^{k+1}, x^k)+2^{p-1}L\nbr{x^{k+1}-x^k}^2.
\end{equation}

Отметим, что здесь мы полагаем $f$ точно заданной, то есть $f_{\delta}=f$ ($\delta=0$) и $ \psi(y, x) = \langle \nabla f(x), y-x \rangle$, где $\nabla f$~--- некоторый субградиент $f$. В работе \cite{StonyakinTrudy} доказано, что $p$ можно выбрать так, чтобы $2^p \geqslant 1+\frac{16\Delta^2}{\varepsilon L}$ и получен следующий результат.

\begin{theorem}\label{thm2ston}
Пусть $\mu = 0$. Для выхода $\hat{y}^{k+1}$ модифицированного алгоритма \ref{alg3} c учетом дополнительной процедуры \eqref{eq12} неравенство $f(\hat{y}^{k+1})-f^*\leqslant\varepsilon$ будет гарантированно выполнено не более, чем после
\begin{equation}\label{eq16fin}
\left\lceil\frac{4LR^2}{\varepsilon}+\frac{64\Delta^2R^2}{\varepsilon^2}\right\rceil\cdot\left\lceil\log_2 \left(1+\frac{16\Delta^2}{\varepsilon L}\right) \right\rceil
\end{equation}
вычислений субградиента $f$.
\end{theorem}

Пусть теперь $\mu>0$. Если положить $\delta_{i+1}=\delta=0$ для всякого $i=\overline{0,k}$ и $\Delta_{i+1}\nbr{x^{i+1}-x^i}\leqslant\frac{\varepsilon}{2}$, то согласно теореме \ref{th03}
$$f(y^{k+1})-f(x_*)\leqslant 2СL\br{1-\frac{\mu}{2СL}}^{k+1}V(x_*,x^0)+\frac{\varepsilon}{2}.$$

В частности, данное неравенство позволяет при понимании величины $\Delta$ в \eqref{eq44} как характеристики негладкости функционала $f$ применить процедуру, аналогичную \eqref{eq12} (при этом уже необходимо требовать $1$-сильную выпуклость прокс-функции в определении \ref{def02}). Тогда  $\varepsilon$-точность решения задачи минимизации $f$ будет достигаться за
$$
O\left(\frac{1}{\varepsilon}\log_2^2\frac{1}{\varepsilon}\right)
$$
шагов градиентного метода (пункта 2 листинга алгоритма \ref{alg3}). Действительно, если положить $\widehat{\delta}=\delta=0$ (то есть $f_{\delta}=f$) и $V(x_*, x^0)\leqslant R^2$, то для достижения качества решения
$$f(y^{k+1})-f(x_*)\leqslant\varepsilon$$
необходимо выполнить не более
$$\left\lceil\frac{2СL}{\mu}\ln\frac{4СLR^2}{\varepsilon}\right\rceil-1$$
итераций алгоритма \ref{alg3}. Ввиду оценки \eqref{EstimP} c учетом процедуры \eqref{eq12} на каждом шаге итоговое число обращений к (суб)градиенту целевого функционала для алгоритма \ref{alg3} можно оценить как $$\left\lceil\br{\frac{2СL}{\mu}+\frac{32C\Delta^2}{\mu\varepsilon}}\ln\br{\frac{4СLR^2}{\varepsilon}+\frac{64СR^2}{\varepsilon^2}}\right\rceil\left\lceil\log_2\br{1+\frac{16\Delta^2}{\varepsilon L}}\right\rceil-1.$$

Покажем, что даст применение похожей схемы для быстрого градиентного метода на классе выпуклых функционалов $f$, для которых при некоторых $L>0$ и $\Delta>0$ верно
\begin{equation}\label{eq32}
f(y)\leqslant f(x)+ \psi(y,x)+\frac{L}{2}\|y-x\|^2+\Delta\|y-x\|\;\;\forall x,y\in Q.
\end{equation}
В частности, для задач выпуклой негладкой оптимизации $\psi(y, x) = \langle \nabla f(x), y - x \rangle $ для некоторого субградиента $\nabla f$.

Рассмотрим вариант быстрого градиентного метода (алгоритм \ref{alg2}), который использует концепцию $(\delta, \Delta, L)$-модели целевого функционала в запрошенной точке. Применим неадаптивный вариант этого метода с постоянным шагом $L_{k+1}=2^pL$ для некоторого $p\in\mathbb{N}$ и $\delta=\frac{\varepsilon}{2\gamma}$ для некоторых $\varepsilon>0$ и фиксированной постоянной $\gamma>0$. Будем при этом считать $f_{\delta}=f$ (то есть функционал $f$ задан точно) и подбирать $p$ так, чтобы на $(k+1)$-ой итерации $(k=0,1,2,\ldots)$ заведомо выполнялось неравенство
$$f(x^{k+1})\leqslant f(y^{k+1})+\psi(x^{k+1},y^{k+1})+\frac{2^pL}{2}\|x^{k+1}-y^{k+1}\|^2+\frac{\varepsilon}{2\gamma}$$
и после $N$ итераций (это можно проверить аналогично \cite{s2})
\begin{equation}\label{eq33}
f(x^N)-f^* \leqslant \frac{8\cdot2^pLR^2}{(N+1)^2}+\frac{\varepsilon N}{2\gamma}.
\end{equation}
Далее, выберем натуральное $p$ так, чтобы гарантированно выполнялась альтернатива
\begin{equation}\label{eq34}
\left[
\begin{array}{ccc}
\Delta\|x^{k+1}-y^{k+1}\|\leqslant\frac{\varepsilon}{2\gamma},\\
\frac{(2^p-1)L}{2}\|x^{k+1}-y^{k+1}\|^2\geqslant\Delta\|x^{k+1}-y^{k+1}\|.
\end{array}
\right.
\end{equation}
Если $$\Delta\|x^{k+1}-y^{k+1}\|>\frac{\varepsilon}{2\gamma},$$
то
$$\frac{(2^p-1)L}{2}\|x^{k+1}-y^{k+1}\|>\frac{(2^p-1)L\varepsilon}{4\gamma\Delta}.$$
Тогда второе неравенство альтернативы \eqref{eq34} заведомо выполнится при
\begin{equation}\label{EstimP}
2^p>1+\frac{4\gamma\Delta^2}{L\varepsilon}.
\end{equation}
Поэтому положим
\begin{equation}\label{eq35}
p = \left\lceil\log_2\br{1+\frac{4\gamma\Delta^2}{L\varepsilon}}\right\rceil.
\end{equation}
Теперь покажем, каким можно выбрать количество итераций $N$, чтобы гарантированно \eqref{eq32} обеспечивало $f(x^N)-f^*\leqslant\varepsilon.$
Для этого потребуем выполнения неравенств
$$\frac{2^{p+3}LR^2}{(N+1)^2}\leqslant\frac{\varepsilon}{2}\text{ и }\frac{\varepsilon N}{2\gamma}\leqslant\frac{\varepsilon}{2},$$
откуда $\gamma\geqslant N$ и $(N+1)^2\geqslant\frac{2^{p+4}LR^2}{\varepsilon}.$
Для упрощения выкладок усилим последнее требование: $N^2\geqslant\frac{2^{p+4}LR^2}{\varepsilon}$, откуда
$$N^2>\frac{16LR^2}{\varepsilon}\br{1+\frac{4\gamma\Delta^2}{L\varepsilon}}\geqslant\frac{16LR^2}{\varepsilon}+\frac{64N\Delta^2R^2}{\varepsilon^2}.$$
Это означает, что $N$ можно выбирать как $\lceil N_2\rceil$, где $N_2$~--- больший корень уравнения
$$N^2-\frac{64\Delta^2R^2}{\varepsilon}N-\frac{16LR^2}{\varepsilon}=0:$$
$$N_2=\frac{32\Delta^2R^2}{\varepsilon^2}+\sqrt{\br{\frac{32\Delta^2R^2}{\varepsilon^2}}^2+\frac{16LR^2}{\varepsilon}}.$$
Далее, в силу неравенства $\sqrt{a+b}>\sqrt{\frac{a}{2}}+\sqrt{\frac{b}{2}}$ ($a,b>0$) имеем $$N \geqslant N_2\geqslant\frac{(32+16\sqrt{2})\Delta^2R^2}{\varepsilon^2}+\frac{2R\sqrt{2L}}{\sqrt{\varepsilon}}.$$

Таким образом, можно гарантировать достижение $\varepsilon$-точного решения задачи минимизации $f$ за
$$O\br{\sqrt{\frac{L}{\varepsilon}}}+O\br{\frac{\Delta^2}{\varepsilon^2}}$$
итераций быстрого градиентного метода, то есть обоснован некоторый аналог результата о слайдинге \cite{Lan_2016}, но уже целевая функция не обязательно есть сумма гладкого и негладкого слагаемых.

При этом возможно использовать метод с адаптивной настройкой констант $L$, а также $\Delta_{k+1} \leqslant \Delta$. Тогда оценка числа итераций может измениться в постоянное число раз. Также при этом за счёт $p$ дополнительных процедур адаптивного подбора $L_{k+1}$ (при условии $L\leqslant L_{k+1}\leqslant2L$ для всякого $k=0,1,\ldots,N-1$) на каждой итерации метода добавится множитель \eqref{eq35}:
$$
p \geqslant \left\lceil\log_2\br{1+\frac{4N\Delta^2}{L\varepsilon}}\right\rceil \geqslant
\log_2\br{1+\frac{(128+64\sqrt{2})\Delta^4R^2}{L\varepsilon^3} + \frac{8R\Delta^2\sqrt{2}}{\sqrt{L\varepsilon^3}}}.
$$

\begin{theorem}\label{thm3ston}
Для выхода $x^N$ модифицированного алгоритма \ref{alg2} c учетом дополнительной процедуры \eqref{eq12} при $\Delta_{k+1} = \Delta$ на (k+1)-й итерации ($k=0,1,\ldots,N-1$)неравенство $f(x^N)-f^*\leqslant\varepsilon$ будет гарантированно выполнено не более, чем после
\begin{equation}\label{eq16finn}
\left\lceil\frac{(32+16\sqrt{2})\Delta^2R^2}{\varepsilon^2}+\frac{2R\sqrt{2L}}{\sqrt{\varepsilon}}\right\rceil\cdot \left\lceil\log_2\br{1+\frac{(128+64\sqrt{2})\Delta^4R^2}{L\varepsilon^3} + \frac{8R\Delta^2\sqrt{2}}{\sqrt{L\varepsilon^3}}}\right\rceil.
\end{equation}
вычислений субградиента $f$.
\end{theorem}

\begin{remark}
Если не предполагать, что на (k+1)-й итерации ($k=0,1,\ldots,N-1$) модифицированных алгоритмов \ref{alg3} и \ref{alg2} выполнено неравенство $L\leqslant L_{k+1}\leqslant2L$ и предусмотреть полностью адаптивную настройку параметров $L$ и $\Delta$ в методе, то оценки \eqref{eq16fin} и \eqref{eq16finn} могут увеличиться не более, чем в
$$
\left\lceil\max\left\{\frac{2L}{L_0}, \frac{2\Delta}{\Delta_0}\right\}\right\rceil
$$
раз. Отметим также, что логарифмические множители в \eqref{eq16fin} и \eqref{eq16finn} можно опустить, если рассматривать неадаптивные варианты методов \ref{alg3} и \ref{alg2} с фиксированным параметром $L_{k+1} = 2^{p}L$ при подходящем натуральном p.
\end{remark}

\section{Адаптивная настройка на величину погрешностей для вариационных неравенств и седловых задач}
Метод с адаптивной настройкой погрешностей можно предложить и для вариационных неравенств \cite{Stonyakin_1907.13455v4}, а также седловых задач. В данном пункте мы покажем, как это можно сделать в модельной общности. Мы отправляемся от \cite{Stonyakin_2019_uroran-25-2}, где введена концепция неточной модели для вариационных неравенств и седловых задач, но с постоянной неточностью и без адаптивной настройки к величинам погрешностей.

Напомним постановку задачи решения вариационного неравенства, а также необходимые понятия и результаты. Для оператора $G: Q \rightarrow \mathbb{R}^n$, заданного на выпуклом компакте $Q\subset \mathbb{R}^n$ под {\it сильным вариационным неравенством} понимаем неравенство вида
\begin{equation}\label{Old_1}
 \langle G(x_*),\ x_*-x\rangle \leqslant 0.
\end{equation}
Отметим, что в \eqref{Old_1} требуется найти $x_* \in Q$ (это $x_*$ и называется решением ВН) для которого
\begin{equation}\label{Old_2}
 \max_{x\in Q}\langle G(x_*),\ x_*-x\rangle \leqslant 0.
\end{equation}
Для монотонного оператора поля $G$ можно рассматривать {\it слабые вариационные неравенства}
\begin{equation}\label{Old_11}
 \langle G(x),\ x_*-x\rangle \leqslant 0.
\end{equation}
Обычно в \eqref{Old_11} требуется найти $x_* \in Q$, для которого \eqref{Old_11} верно при всех $x \in Q$.

Предложен аналог концепции неточной модели целевой функции в оптимизации для вариационных неравенств и седловых задач. Для удобства будем рассматривать задачу нахождения решения $x_*\in Q$ абстрактной задачи равновесного программирования
\begin{equation}\label{eq13}
\psi(x,x_*)\geqslant 0 \quad \forall x\in Q
\end{equation}
для некоторого выпуклого компакта $Q\subset\mathbb{R}^n$, а также функционала $\psi:Q\times Q\rightarrow\mathbb{R}$. Если предположить абстрактную монотонность функционала $\psi$:
\begin{equation}\label{eq14}
\psi(x,y)+\psi(y,x)\leqslant0\;\;\forall x,y\in Q,
\end{equation}
то всякое решение \eqref{eq13} будет также и решением двойственной задачи равновесия
\begin{equation}\label{eq115}
\psi(x_*,x)\leqslant 0 \quad \forall x\in Q.
\end{equation}
В общем случае сделаем предположение о существовании решения $x_*$ задачи \eqref{eq13}. Приведем пару примеров задания $\psi$, для которых данное условие заведомо выполнено.

\begin{example}
Если для некоторого оператора $G:Q\rightarrow\mathbb{R}^n$ положить
\begin{equation}\label{eq16}
\psi(x,y)=\langle G(y),x-y\rangle\;\;\forall x,y\in Q,
\end{equation}
то \eqref{eq13} и \eqref{eq115} будут равносильны соответственно стандартным сильному и слабому вариационному неравенству с оператором $G$.
\end{example}

\begin{example}\label{eq: Mixed_VI}
Для некоторого оператора $G:Q\rightarrow\mathbb{R}^n$ и выпуклого функционала $h:Q\rightarrow\mathbb{R}^n$ простой структуры выбор функционала
\begin{equation}\label{eq17}
\psi(x,y)=\langle G(y),x-y\rangle+h(x)-h(y)
\end{equation}
приводит к {\it смешанному вариационному неравенству}
\begin{equation}\label{eq18}
\langle G(y),y-x\rangle+h(y)-h(x)\leqslant0,
\end{equation}
которое в случае монотонности оператора $G$ влечет
\begin{equation}\label{eq19}
\langle G(x),y-x\rangle+h(y)-h(x)\leqslant0.
\end{equation}
\end{example}

Концепцию $(\delta, \Delta, L)$-модели для выделенного выше класса задач \eqref{eq13} --- \eqref{eq115} возможно ввести следующим образом.

\begin{definition}\label{Def_Model_VI}
Будем говорить, что функционал $\psi$ допускает $(\delta, \Delta, L)$-модель $\psi_{\delta} (x, y)$ при некоторых фиксированных значениях параметров $L,\;\delta,\;\Delta>0$ в произвольной точке $y$ относительно дивергенции Брэгмана $V(y,x)$, если для произвольных $x, y, z \in Q$ верны:
\begin{enumerate}
\item[(i)] $\psi(x, y) \leqslant \psi_{\delta}(x, y) + \delta$;
\item[(ii)] функционал $\psi_{\delta} (x, y)$ выпуклый по первой переменной и $\psi_{\delta}(x,x)=0$;
\item[(iii)] ({\it абстрактная $\delta$-монотонность})
\begin{equation}\label{eq:abstr_monot}
\psi_{\delta}(x,y)+\psi_{\delta}(y,x)\leqslant \delta;
\end{equation}
\item[(iv)] ({\it обобщённая относительная гладкость})
\begin{equation}\label{eq20}
\psi_{\delta}(x,y)\leqslant\psi_{\delta}(x,z)+\psi_{\delta}(z,y)+ LV(x,z)+ LV(z,y)+\Delta\nbr{y-z} + \delta.
\end{equation}
\end{enumerate}
\end{definition}

Естественно возникает идея обобщить этот метод на абстрактные задачи \eqref{eq13} и \eqref{eq115} в предположениях их разрешимости, а также (i)--(iv). При этом будем учитывать погрешность $\delta$ в \eqref{eq20}, а также погрешность $\tilde{\delta}$ решения вспомогательных задач на итерациях согласно одному из достаточно известных в алгоритмической оптимизации подходов:
\begin{equation}\label{eq22}
x:=\argmin\limits_{y\in Q}^{\tilde{\delta}}\varphi(y),\text{ если }\langle\nabla\varphi(x),x-y\rangle\leqslant\tilde{\delta}.
\end{equation}
Во многих случаях можно считать \cite{s2}, что предполагается отрешивание вспомогательных задачи минимизации с некоторой точностью $\tilde{\delta}$ по функции.

К выделенному классу задач \eqref{eq13} и \eqref{eq115}, допускающих существование ($\delta, \Delta, L$)-модели, возможно применить аналог проксимального зеркального метода А.С. Немировского с адаптивным выбором шага.
Опишем $(N+1)$-ую итерацию этого метода ($N=0,1,2,\ldots$), выбрав начальное приближение
$x^0=\argmin\limits_{x\in Q}d(x),$
зафиксировав точность $\varepsilon>0$, а также некоторые константы $L_0\leqslant2L$ и $\Delta_0\leqslant2\Delta$.
\begin{algorithm}[h!]\label{alg:SIGM}
\caption{Адаптивный метод для концепции ($\delta, \Delta, L$)-модели.}
\begin{itemize}
\item[1.] $N:=N+1$,\;$L_{N+1}:=\frac{L_N}{2},\;\Delta_{N+1}:=\frac{\Delta_N}{2}$.
\item[2.] Вычисляем:\\
$y^{N+1}:=\argmin\limits_{x\in Q}^{\tilde{\delta}}\fbr{\psi_{\delta}(x, x^N)+L_{N+1}V(x,x^N)}$,\\
$x^{N+1}:=\argmin\limits_{x\in Q}^{\tilde{\delta}}\fbr{\psi_{\delta}(x, y^{N+1})+L_{N+1}V(x,x^N)}$\\
до тех пор, пока не будет выполнено:
\begin{equation}\label{eq23}
\begin{split}
\psi_{\delta}(x^{N+1}, x^N)\leqslant\psi_{\delta}(y^{N+1}, x^N)+\psi_{\delta}(x^{N+1}, y^{N+1})+\\
+L_{N+1}V(y^{N+1}, x^N)+L_{N+1}V(y^{N+1}, x^{N+1})+\Delta_{N+1}\nbr{y^{N+1}-x^{N+1}}+\delta.
\end{split}
\end{equation}
\item[3.] \textbf{Если} \eqref{eq23} не выполнено, \textbf{то} $L_{N+1}:=2L_{N+1}$, $\Delta_{N+1}:=2\Delta_{N+1}$ и повторяем п. 2.
\item[4.] \textbf{Иначе} переход к п.~1.
\item[5.] Критерий остановки метода:
\begin{equation}\label{eq24}
\sum\limits_{k=0}^{N-1}\frac{1}{L_{k+1}}\geqslant\frac{\max_{x \in Q} V(x,x^0)}{\varepsilon}.
\end{equation}
\end{itemize}
\end{algorithm}

Для краткости будем всюду далее обозначать
\begin{equation}\label{eq25}
S_N:=\sum\limits_{k=0}^{N-1}\frac{1}{L_{k+1}}.
\end{equation}

Справедлива следующая
\begin{theorem}\label{fedyorth1}
После остановки рассматриваемого метода для всякого $x \in Q$ будет заведомо выполнено неравенство:
\begin{equation}\label{eq26}
-\frac{1}{S_N}\sum\limits_{k=0}^{N-1}\frac{\psi_{\delta}(x,y^{k+1})}{L_{k+1}}\leqslant \varepsilon+2\tilde{\delta}+\delta+\frac{1}{S_N}\sum\limits_{k=0}^{N-1}\frac{\Delta_{k+1}\nbr{y^{k+1}-x^{k+1}}}{L_{k+1}},
\end{equation}
a также
\begin{equation}\label{eq27}
\psi(\tilde{y},x)\leqslant \varepsilon+2\tilde{\delta}+3\delta + \frac{1}{S_N}\sum\limits_{k=0}^{N-1}\frac{\Delta_{k+1}\nbr{y^{k+1}-x^{k+1}}}{L_{k+1}}
\end{equation}
при
\begin{equation}\label{eq28}
\tilde{y}:=\frac{1}{S_N}\sum\limits_{k=0}^{N-1}\frac{y^{k+1}}{L_{k+1}}.
\end{equation}
\end{theorem}

\begin{remark}
Для обычных слабых вариационных неравенств \eqref{Old_11} неравенство \eqref{eq27} можно заменить на
\begin{equation}\label{eq14krit}
\max_{x\in Q} \left \langle G(x), \tilde{y} -x \right\rangle \leqslant \varepsilon + 2 \widetilde{\delta} + 3\delta + \frac{1}{S_N}\sum\limits_{k=0}^{N-1}\frac{\Delta_{k+1}\nbr{y^{k+1}-x^{k+1}}}{L_{k+1}}.
\end{equation}
\end{remark}

Аналогично теореме \ref{thm2ston} можно сформулировать результат о том, что при $\gamma=\delta=\widetilde{\delta}=0$ процедурой типа \eqref{eq12} можно добиться выполнения $\max\limits_{x\in Q}\abr{G(x),\widetilde{y}-x}\leqslant\varepsilon$ за $$\left\lceil\frac{4LR^2}{\varepsilon}+\frac{64\Delta^2}{\varepsilon^2}\right\rceil\left\lceil\log\br{1+\frac{16\Delta^2}{\varepsilon L}}\right\rceil$$
обращения к оракулу для $G$.

Можно ввести также аналог концепции $(\delta, \Delta, L)$-модели для седловых задач. Напомним, что постановка седловой задачи предполагает, что для выпуклого по $u$ и вогнутого по $v$ функционала $f(u,v):\mathbb{R}^{n_1+n_2}\rightarrow\mathbb{R}$ ($u\in Q_1\subset\mathbb{R}^{n_1}$ и $v\in Q_2\subset\mathbb{R}^{n_2}$) требуется найти $(u_*,v_*)$ такую, что:
\begin{equation}\label{eq31}
f(u_*,v)\leqslant f(u_*,v_*)\leqslant f(u,v_*)
\end{equation}
для произвольных $u\in Q_1$ и $v\in Q_2$. Мы считаем $Q_1$ и $Q_2$ выпуклыми компактами в пространствах $\mathbb{R}^{n_1}$ и $\mathbb{R}^{n_2}$ и поэтому $Q=Q_1\times Q_2\subset\mathbb{R}^{n_1+n_2}$ также есть выпуклый компакт. Для всякого $x=(u,v)\in Q$ будем полагать, что
$||x||=\sqrt{||u||_1^2+||v||_2^2},$
где $||\cdot||_1$ и $||\cdot||_2$~--- нормы в пространствах $\mathbb{R}^{n_1}$ и $\mathbb{R}^{n_2}$). Условимся обозначать $x=(u_x,v_x),\;y=(u_y,v_y)\in Q$.

Хорошо известно, что для достаточно гладкой функции $f$ по $u$ и $v$ задача \eqref{eq31} сводится к вариационному неравенству с оператором
\begin{equation}\label{eq32}
G(x)=
\begin{pmatrix}
f_u'(u_x,v_x)\\
-f_v'(u_x,v_x)
\end{pmatrix}.
\end{equation}

Предложим некоторую вариацию концепции $(\delta, \Delta, L)$-модели для вариационных неравенств, но уже на более узком классе седловых задач.

\begin{definition}\label{Def_Sedlo}
Будем говорить, что для некоторой постоянной $\delta > 0$ функция $\psi_{\delta}(x,y)\;(\psi:\mathbb{R}^{n_1+n_2}\times\mathbb{R}^{n_1\times n_2}\rightarrow\mathbb{R})$ есть {\bf $(\delta, \Delta, L)$-модель} для седловой задачи \eqref{eq31}, если для некоторого функционала $\psi_{\delta}$ при произвольных $x, y, z \in Q$ выполнены предположения (i) -- (iv) определения \ref{Def_Model_VI}, а также справедливо неравенство:
\begin{equation}\label{eq33}
f(u_y,v_x)-f(u_x,v_y)\leqslant-\psi_{\delta}(x,y) + \delta \quad \forall x,y\in Q.
\end{equation}
\end{definition}

Из теоремы \ref{fedyorth1} вытекает
\begin{theorem}\label{th_sedlo}
Если для седловой задачи \eqref{eq31} существует $(\delta, \Delta, L)$-модель $\psi_{\delta}(x,y)$, то после остановки алгоритма 3 получаем точку
\begin{equation}\label{eq36}
\tilde{y}=(u_{\tilde{y}},v_{\tilde{y}}):=(\tilde{u},\tilde{v}):=\frac{1}{S_N}\sum\limits_{k=0}^{N_1}\frac{y^{k+1}}{L_{k+1}},
\end{equation}
для которой верна оценка величины-качества решения седловой задачи:
\begin{equation}\label{eq37}
\max_{v\in Q_2}f(\tilde{u},v)-\min_{u\in Q_1}f(u,\tilde{v})\leqslant \varepsilon + \frac{1}{S_N}\sum\limits_{k=0}^{N-1}\frac{\Delta_{k+1}\nbr{y^{k+1}-x^{k+1}}}{L_{k+1}} + 2\tilde{\delta}+2\delta.
\end{equation}
\end{theorem}

\section{Aдаптивные алгоритмы зеркального спуска для негладких задач выпуклой оптимизации с функциональными ограничениями}

В данном разделе мы рассмотрим некоторые методы для задачи минимизации выпуклой негладкой функции $f(x)\rightarrow \min $ с липшицевым выпуклым функциональным ограничением $g(x)\leqslant0$. Далее будем полагать, что $x_*$~--- одно из решений такой задачи. Мы будем рассматривать для выделенного класса задач методы, аналогичные схемам с переключениями, которые восходят к пионерским работам \cite{Polyak_1967,Shor_1967}. Недавно в \cite{Bayandina_2018_lectures} были предложены методы зеркального спуска такого типа для условных задач с адаптивным выбором шага и критериев остановки. Более того, оказалось \cite{Ivanova_1911.07354v2,Stonyakin_1908.00218v4}, что для указанных процедур оптимальную на классе выпуклых липшицевых целевых функционалов оценку скорости сходимости $O(\varepsilon^{-2})$ можно распространить и на класс квазивыпуклых гёльдеровых целевых функционалов. Рассмотрены и приложения к задаче оптимизации компьютерной сети для логарифмического целевого функционала \cite{Ivanova_1911.07354v2}. По итогам проведённых экспериментов выяснилось, что как правило, быстрее всех из рассмотренных схем \cite{Bayandina_2018_lectures,Ivanova_1911.07354v2,Stonyakin_1908.00218v4} работает следующий метод.
\begin{algorithm}
\caption{Адаптивный зеркальный спуск}
\label{alg4}
\begin{algorithmic}[1]
\REQUIRE $\varepsilon>0,\;\Theta_0:d(x_*)\leqslant\Theta_0^2$
\STATE $x^0=argmin_{x\in Q}\,d(x)$
\STATE $I=:\emptyset$
\STATE $k\leftarrow0$
\REPEAT
  \IF{$g(x^k)\leqslant\varepsilon||\nabla g(x^k)||_*$}
    \STATE $h_k\leftarrow\frac{\varepsilon}{||\nabla f(x^k)||_{*}^2}$
    \STATE $x^{k+1}\leftarrow Mirr_{x^k}(h_k\nabla f(x^k))\;\text{// \emph{"продуктивные шаги"}}$
    \STATE $k\rightarrow I$
  \ELSE
    \STATE $h_k\leftarrow\frac{\varepsilon}{||\nabla g(x^k)||_{*}}$
    \STATE $x^{k+1}\leftarrow Mirr_{x^k}(h_k\nabla g(x^k))\;\text{// \emph{"непродуктивные шаги"}}$
  \ENDIF
  \STATE $k\leftarrow k+1$
\UNTIL{
  \begin{equation}\label{eq1}
  \frac{2\Theta_0^2}{\varepsilon^2} \leqslant\sum\limits_{k\in I}\frac{1}{||\nabla f(x^k)||_{*}^2}+|J|,
  \end{equation}
  }
где $|J|$~--- количество непродуктивных шагов (мы обозначим через $|I|$ количество продуктивных шагов, то есть $|I|+|J|=N$).
\ENSURE{$\widehat{x}=\frac{1}{\sum\limits_{k\in I}h_k}\sum\limits_{k\in I}h_kx^k$.}
\end{algorithmic}
\end{algorithm}

Обозначим через $\varepsilon>0$ фиксированную точность, $x_0$~--- начальное приближение такое, что для некоторого $\Theta_0>0$ верно неравенство $V(x_*,x^0)\leqslant \Theta_0^2$. Пусть
\begin{equation}
|g(x)-g(y)|\leqslant M_g||x-y||\;\forall x,y\in Q.
\end{equation}
Тогда справедлив следующий результат об оценке качества найденного решения предложенного метода (в \cite{Ivanova_1911.07354v2} имеется полное доказательство в случае евклидовой прокс-структуры, для произвольной прокс-структуры это доказательство аналогично).

\begin{theorem}\label{th1}
После остановки предложенного алгоритма \ref{alg4} справедливы неравенства:
$$f(\widehat{x})-f(x_*)\leqslant\varepsilon\text{ и } g(\widehat{x})\leqslant \varepsilon M_g.$$
\end{theorem}

На базе теоремы \ref{th1} в предположении липшицевости целевого функционала
\begin{equation}
|f(x)-f(y)|\leqslant M_f||x-y||
\end{equation}
можно оценить количество итераций, необходимых для выполнения критерия остановки \eqref{eq1}. Ясно, что $\forall k\in I\;||\nabla f(x^k)||_*\leqslant M_f$ и поэтому
$$|J|+\sum\limits_{k\in I}\frac{1}{||\nabla f(x^k)||_*^2}\geqslant|J|+\frac{|I|}{M_f^2}\geqslant(|I|+|J|)\frac{1}{\max\{1,M_f^2\}}.$$
Это означает, что при
$$N\geqslant\frac{2\Theta_0^2\max\{1,M_f^2\}}{\varepsilon^2}$$
критерий остановки \eqref{eq1} заведомо выполнен, то есть искомая точность достигается за $O\br{\frac{1}{\varepsilon^2}}$ итераций.

По аналогии с (\cite{Bayandina_2018_lectures}, теорема 3.2) можно проверить следующий результат, означающий прямо-двойственность алгоритма \ref{alg4} в случае ограниченного множества $Q$.

\begin{theorem}
Пусть $Q$ ограничено и $\displaystyle\max_{x\in Q}d(x)\leqslant\Theta_{0}^{2}$. Тогда если применить алгоритм \ref{alg4} к задаче
$$\varphi(\lambda)=\min_{x\in Q}\left\{f(x)+\sum\limits_{i=1}^{m}\lambda_{i}g_{i}(x)\right\}\rightarrow\max_{\lambda_{i}\geqslant0,\,\,i=1,\ldots, m},$$
то после остановки для найденной пары $(\widehat{x}^{k},\widehat{\lambda}^{k})$ будет верно:
$$f(\widehat{x}^{k})-\varphi(\widehat{\lambda}^{k})\leqslant\varepsilon,\quad g(\widehat{x}^{k})\leqslant M_{g}\varepsilon.$$
\end{theorem}

Похожий на алгоритм \ref{alg4} метод можно предложить и для задач со следующими релаксациями известного неравенства Коши-Буняковского для субградиентов $\nabla f(x)$ и $\nabla g(x)$:
\begin{equation}\label{eq40}
\abr{\nabla f(x),x-y}\leqslant\omega\nbr{\nabla f(x)}_*\sqrt{2V(y,x)} \quad \forall y \in Q
\end{equation}
и произвольного $x\in Q:\;V(x_*,x)\geqslant\frac{\varepsilon^2}{2}$, некоторой фиксированной постоянной $\omega>0$ (здесь $x_*$ --- ближайшее решение к начальной точке $x_0$ с точки зрения дивергенции $V$), а также
\begin{equation}\label{eq41}
\abr{\nabla g(x),x-y}\leqslant M_g\sqrt{2V(y,x)}
\end{equation}
для константы $M_g>0$.

Ясно, что при $V(y,x)\geqslant\frac{1}{2}\nbr{y-x}^2$ неравенство \eqref{eq40} верно для $\omega=1$. Неравенство \eqref{eq41}, в частности, верно в случае относительной липшицевости $g$ \cite{Lu_1710.04718v3}. Рассмотрим следующий алгоритм 5.

\begin{algorithm}
\caption{Адаптивный зеркальный спуск}
\label{alg5}
\begin{algorithmic}[1]
\REQUIRE $\varepsilon>0,\;\Theta_0:d(x_*)\leqslant\Theta_0^2$
\STATE $x^0=argmin_{x\in Q}\,d(x)$
\STATE $I=:\emptyset$
\STATE $k\leftarrow0$
\REPEAT
  \IF{$g(x^k)\leqslant\varepsilon M_g$}
    \STATE $h_k^f\leftarrow\frac{\varepsilon}{||\nabla f(x^k)||_{*}^2}$
    \STATE $x^{k+1}\leftarrow Mirr_{x^k}(h_k\nabla f(x^k))\;\text{// \emph{"продуктивные шаги"}}$
    \STATE $k\rightarrow I$
  \ELSE
    \STATE $h_k^g\leftarrow\frac{\varepsilon}{M_g}$
    \STATE $x^{k+1}\leftarrow Mirr_{x^k}(h_k\nabla g(x^k))\;\text{// \emph{"непродуктивные шаги"}}$
  \ENDIF
  \STATE $k\leftarrow k+1$
\UNTIL{
  \begin{equation}\label{eqstoprule}
  \frac{2\Theta_0^2}{\varepsilon^2} \leqslant\sum\limits_{k\in I}\frac{\omega^2}{||\nabla f(x^k)||_{*}^2}+|J|,
  \end{equation}
  }
где $|J|$~--- количество непродуктивных шагов (мы обозначим через $|I|$ количество продуктивных шагов, то есть $|I|+|J|=N$).
\ENSURE{$\widehat{x}=\frac{1}{\sum\limits_{k\in I}h_k^f}\sum\limits_{k\in I}h_k^fx^k$.}
\end{algorithmic}
\end{algorithm}

Тогда из \eqref{eq40} и \eqref{eq41} имеем:
$$h_k\abr{\nabla f(x^k),x^k-x_*}\leqslant\frac{\omega^2\varepsilon^2}{2\nbr{\nabla f(x^k)}_*^2}+V(x_*,x^k)-V(x_*,x^{k+1})\;\;\forall k\in I,$$
$$\varepsilon^2< h_k g(x^k) \leqslant h_k\abr{\nabla g(x^k),x^k-x_*}\leqslant\frac{\varepsilon^2}{2}+V(x_*,x^k)-V(x_*,x^{k+1})\;\;\forall k\in J.$$
После суммирования указанных неравенств получаем:
$$\sum\limits_{k\in I}h_k\br{f(x^k)-f(x_*)}\leqslant\frac{\omega^2\varepsilon}{2}\sum\limits_{k\in I}h_k-\frac{\varepsilon^2}{2}|J|+V(x_*,x^0).$$

\begin{theorem}
После выполнения критерия остановки \eqref{eqstoprule} справедлива оценка:
\begin{equation}\label{eq42}
f(\widehat{x})-f(x_*)\leqslant\omega^2\varepsilon \text{  и  }g(\widehat{x})\leqslant\varepsilon M_g
\end{equation}
или $V(x_*,x^k)<\frac{\varepsilon^2}{2}$ для некоторого $k$.
\end{theorem}

\begin{remark}\label{StonRem8}
Вместо условия \eqref{eq40} можно рассмотреть и неравенство $\abr{\nabla f(x),x-y}\leqslant M_f\sqrt{2V(y,x)}$, которое верно в случае относительной липшицевости $f$ \cite{Lu_1710.04718v3}. Для этого необходимо выбирать в алгоритме \ref{alg5} продуктивные шаги $h_k^f = \frac{\varepsilon}{M_f^2}$, а также критерий остановки
$$2V(x_*,x^0) \leqslant 2\Theta_0^2\leqslant \frac{\varepsilon^2|I|}{M_f^2}+\varepsilon^2|J|.$$
После выполнения этого критерия будет верно неравенство
\begin{equation}\label{eq42}
f(\widehat{x})-f(x_*)\leqslant \varepsilon \text{ и }g(\widehat{x})\leqslant\varepsilon M_g.
\end{equation}
\end{remark}

Полученный результат позволяет применить предложенную алгоритмическую схему для негладких задач выпуклой однородной оптимизации с относительной точностью. Такая постановка восходит к работам Ю.\,Е.\,Нестерова (см. главу 6 диссертации \cite{Nesterov_2013a}). Как показано Ю.\,Е.\,Нестеровым, подход к оценке качества решения задачи с точки зрения именно относительной точности вполне оправдан для разных прикладных задач (линейное программирование, проектирование механических конструкций, задача оптимальной эллипсоидальной аппроксимации выпуклого множества и пр.) при подходящих условиях на желаемую относительную точность. Известно, что достаточно широкий класс задач оптимизации с относительной точностью можно сводить к минимизации выпуклой однородной функции. Итак, рассматривается на выпуклом замкнутом множестве $Q\subset\mathbb{R}^{n}$ задача минимизации выпуклой однородной функции вида
\begin{equation}\label{E1}
f(x)\rightarrow\min_{x\in Q}\end{equation}
с выпуклыми функционалами ограничениями
\begin{equation}\label{E2}
g_{p}(x)\leqslant0,\quad p=\overline{1,m}.
\end{equation}
Стандартно будем обозначать $g(x):=\max_{1\leqslant p\leqslant m}\{g_{p}(x)\}.$

С использованием рассмотренных вариантов концепции относительной липшицевости покажем, как можно выписать оценки сходимости для зеркальных спусков с переключениями в случае, когда $0$ не обязательно есть внутренняя точка субдифференциала $f$ в нуле. Рассмотрен следующий ослабленный вариант этого условия
\begin{equation}\label{EE5}
B_{\gamma_{0}}^{K^{*}}(0)\subseteq\partial f(0)\subseteq B_{\gamma_{1}}^{K^{*}}(0),
\end{equation}
где $K^{*}$~--- сопряженный конус к некоторому полунормированному конусу $K\subset\mathbb{R}^{n}$ с законом сокращения и конус-полунормой $\|\cdot\|_{K}$ (отличие от обычной полунормы в том, что $\|\alpha x\|_{K}=\alpha\|x\|_{K}$ лишь для $\alpha\geqslant 0$). Здесь под {\it сопряженным конусом} $K^{*}$ понимается набор функционалов вида $\psi_{\ell}=\max\{0,\ell(x)\}$ для линейных функционалов $\ell:\,K\rightarrow\mathbb{R}\,:\,\ell(x)\leqslant C_{\ell}\|x\|_{K}$ при некотором $C_{\ell}>0\,\,\forall\,x\in K$. Ясно \cite{Stonyakin_2016, Stonyakin_2018}, что $K^{*}$ будет выпуклым конусом с операциями сложения
$$\psi_{\ell_{1}}\oplus\psi_{\ell_{2}}:=\psi:\quad \psi(x)=\max\{0,\ell_{1}(x)+\ell_{2}(x)\}$$
и умножения на скаляр $\lambda\geqslant0$ $\psi_{\lambda\ell}(x)=\lambda\psi_{\ell}(x)=\lambda\max\{0,\ell(x)\}\quad \forall\,x\in K.$
На $K^{*}$ можно ввести норму
$\|\psi_{\ell}\|_{K^{*}}=\sup_{\|x\|_{K}\leqslant1}\max\{0,\ell(x)\}=\sup_{\|x\|_{K}\leqslant1}\ell(x)$
и шар радиуса $r$ $B_{r}^{K^{*}}(0)=\{\psi_{\ell}\in K^{*}\,|\,\,\|\psi_{\ell}\|_{K^{*}}\leqslant r\}$.

Из аналога теоремы об опорном функционале в нормированных конусах \cite{Stonyakin_2016} получаем, что
\begin{equation}\label{EE6}
\|x\|_{K}=\max_{\psi_{\ell}\in B_{1}^{K^{*}}(0)}\ell(x).
\end{equation}
Для полунормированного конуса при $\|x\|_{K}=0$ достаточно выбрать $\ell\equiv 0$ и \eqref{EE6} будет верно. Приведем некоторый пример пары $(K,K^{*})$.

\begin{example}
Пусть $K=\{(x,y)\,|\,x,y\in\mathbb{R}\}$ и $\|(x,y)\|_{K}=\sqrt{x^{2}+y^{2}}+y$. Можно проверить, что в таком случае
$K^{*}=\{\psi_{\ell}\,|\,\ell\left((x,y)\right)=\lambda x+\mu y\,:\,\mu+\frac{\lambda^{2}}{\mu}<+\infty\,\text{или}\,\lambda=\mu=0\}, \text{ а }$
$$
\|\psi_{\ell}\|_{K^{*}} =
\begin{cases}
0, & \text{если $\lambda=\mu=0$;} \\
\displaystyle\frac{\mu}{2}+\frac{\lambda^{2}}{2\mu}, & \text{если $\mu+\displaystyle\frac{\lambda^{2}}{\mu}<+\infty$ при $\mu>0$.}
\end{cases}
$$
Тогда $B_{1}^{K^{*}}(0)$ имеет вид круга на плоскости $(\lambda,\mu)$ радиуса $1$ с центром в точке $\lambda=0,\,\mu=1$. Не уменьшая общности рассуждений, будем полагать $K=\displaystyle\bigcup_{r\geqslant0}B_{r}^{K}(0)$, а также $x_{*}\in K$ для точного решения $x_{*}$ рассматриваемой задачи минимизации $f$ на $Q$.

Согласно схеме рассуждений (\cite{Nesterov_2013a}, глава 6) для вывода оценок скорости сходимости методов с относительной точностью необходимо знать оценку $R$ величины расстояния от точки старта $x^0$ до ближайшего решения $x_{*}$. Однако в конусах, вообще говоря, не задана операция вычитания и поэтому в качестве аналога нормы разности можно использовать метрику $d^{K}(x^{0},x_{*})$, где
$$d^{K}(x,y)=\sup_{\|\psi_{\ell}\|_{K^{*}}\leqslant 1}|\psi_{\ell}(x)-\psi_{\ell}(y)| \quad \forall x, y \in Q.$$
\end{example}
Некоторые условия, при которых нормированный конус допускает существование метрики такого типа, исследованы автором работы в \cite{Stonyakin_2016,Stonyakin_2018,Stonyakin_2019}.

Получен аналог теоремы 6.1.1 \cite{Nesterov_2013a} для указанного выше предположения ($x_{0},x_{*}\in K$, причём по аналогии с разд. 6 из \cite{Nesterov_2013a} полагаем $x^{0}: \|x^{0}\|_{K} = \min\{\|x\|_{K}, \; x \in Q\}$).
\begin{theorem}\label{th8}
\begin{itemize}
\item[$1)$] $\forall\,x\in K\quad\gamma_{0}\|x\|_{K}\leqslant f(x)\leqslant\gamma_{1}\|x\|_{K}$.
Более того,
$$\frac{\gamma_0}{\gamma_1} f(x^{0})\leqslant\gamma_{0}\|x^{0}\|_{K}\leqslant f(x_{*})\leqslant f(x^{0})\leqslant\gamma_{1}\|x^{0}\|_{K}.$$
\item[$2)$] Для всякого точного решения $x_{*}\in K$ справедливо неравенство:
$$d^{K}(x^{0},x_{*})\leqslant\|x^{0}\|_{K}+\|x_{*}\|_{K}\leqslant\frac{2}{\gamma_{0}}f^{*}\leqslant\frac{2}{\gamma_{0}}f(x^{0}).$$
\end{itemize}
\end{theorem}

Для применимости к поставленной задаче приведенного выше метода зеркального спуска (алгоритм \ref{alg5}) с предложенными вариантами условий относительной липшицевости достаточно выбрать прокс-структуру так, чтобы
\begin{equation}\label{eq43}
V(x_*,x^0)\leqslant\widehat{\omega}d^K(x_*,x^0),
\end{equation}
для некоторой постоянной $\widehat{\omega}>0$.
Этого можно добиться, например, при $x^0=0$. Тогда если положить $d(x)=\nbr{x}_K$, то $V(x_*,0)=\nbr{x_*}_K=d^K(x_*,0)$, то есть \eqref{eq43} выполнено при $\widehat{\omega}=1$.

Поскольку функционал $f$ однороден, то при условии ограниченности субдифференциала $\partial f(0)$ для некоторой константы $M_f>0$ будет верно $\nbr{\nabla f(x)}_*\leqslant M_f$. Поэтому критерий остановки модифицированного алгоритма \ref{alg5} в соответствии с замечанием \ref{StonRem8} заведомо будет выполнен после $2\Theta_0^2 \max\fbr{1, M_f^2}\varepsilon^{-2}$ итераций. Будем полагать, что $\Theta_0^2 \geqslant \widehat{\omega} d^K(x^0,x_*)\geqslant V(x_*,x^0)$ и для некоторого $N$ (число итераций) выберем $\varepsilon=\frac{\Theta_0^2}{\sqrt{N}}$. При данном $\varepsilon$ после остановки модифицированного алгоритма \ref{alg5} будут верны неравенства $f(\widehat{x})-f(x_*)\leqslant\frac{\Theta_0^2}{\sqrt{N}}$ и $g(\widehat{x})\leqslant\frac{M_g\Theta_0^2}{\sqrt{N}}$. Теперь согласно замечанию \ref{StonRem8} имеем:
$$\frac{\Theta_0^2}{\sqrt{N}}\leqslant\frac{2f(x_*)}{\gamma_0\sqrt{N}},\text{ то есть }f(\widehat{x})\leqslant f(x_*)\br{1+\frac{2}{\gamma_0\sqrt{N}}}\text{ и }g(\widehat{x})\leqslant\frac{M_g\Theta_0^2}{\sqrt{N}}.$$
Поэтому для достижения относительной точности $\delta>0$ по функции заведомо достаточно
$N\geqslant\frac{4}{\gamma_0^2\delta^2}$
шагов модифицированного алгоритма \ref{alg5} c постоянными продуктивными шагами согласно замечанию \ref{StonRem8}.
Ясно, что при $2\max\{1, M_f^2\} \leqslant \Theta_0^2$ критерий остановки заведомо выполнен при любом $N$, поскольку в таком случае
$$\frac{2\Theta_0^2\max\fbr{1, M_f^2}}{\frac{\Theta_0^4}{N}}=\frac{2N\max\fbr{1, M_f^2}}{\Theta_0^2} \leqslant N.$$
Поэтому справедлива следующая

\begin{theorem}
Пусть однородный выпуклый $M_f$-липшицев функционал $f$ удовлетворяет предположению замечания \ref{StonRem8} и при этом $2\max\{1, M_f^2\} \leqslant \Theta_0^2$. Тогда после $N \geqslant \frac{4}{\gamma_0^2\delta^2}$ итераций модифицированного алгоритма \ref{alg5} c постоянными продуктивными шагами согласно замечанию \ref{StonRem8} гарантированно будут верны неравенства:
$$
f(\widehat{x}) \leqslant f(x_*)(1 + \delta) \text{   и   } g(\widehat{x}) \leqslant \frac{M_g \Theta_0^2}{\sqrt{N}}.
$$
\end{theorem}

Если не накладывать дополнительных условий на связь $M_f$ и $\Theta_0^2$, то это приведёт к увеличению количества итераций. В общем случае при $\varepsilon=\frac{\Theta_0^2}{\sqrt{N}}$ и $N\geqslant\frac{4}{\gamma_0^2\delta^2}$ критерий остановки заведомо будет выполнен после
$$
\frac{2\Theta_0^2\max\fbr{1, M_f^2}}{\frac{\Theta_0^4}{N}}=\frac{2N\max\fbr{1, M_f^2}}{\Theta_0^2} \geqslant\frac{8\max\fbr{1, M_f^2}}{\gamma_0^2\delta^2}
$$
итераций указанного метода. Это позволяет сформулировать следующий результат
\begin{theorem}
Пусть однородный выпуклый $M_f$-липшицев функционал $f$ удовлетворяет предположению замечания \ref{StonRem8}. Тогда при $N = \frac{4}{\gamma_0^2\delta^2}$ и $\varepsilon=\frac{\Theta_0^2}{\sqrt{N}}$ после не менее, чем
$$\frac{8\max\fbr{1, M_f^2}}{\gamma_0^2\delta^2}$$
итераций модифицированного алгоритма \ref{alg5} c постоянными продуктивными шагами согласно замечанию \ref{StonRem8} гарантированно будут верны неравенства:
$$
f(\widehat{x}) \leqslant f(x_*)(1 + \delta) \text{   и   } g(\widehat{x}) \leqslant \frac{M_g \Theta_0^2}{\sqrt{N}}.
$$
\end{theorem}

\section*{Заключение}

В настоящей статье предложены некоторые адаптивные методы градиентного типа, применимые к задачам негладкой оптимизации.

В первых трёх разделах рассмотрены подходы, основанные на введении погрешностей. Погрешности могут быть как естественными (задание целевой функции и градиента), так и искусственными (сведение негладкой задачи к гладкой с некоторой неточностью). При этом одна из основных отличительных особенностей предлагаемых подходов~--- наличие в оптимизационной модели целевого функционала переменной величины, соответствующей погрешности градиента или степени негладкости задачи.

Получены оценки скорости сходимости для адаптивного неускоренного градиентного метода с адаптивной настройкой на гладкость задачи и величины погрешностей. Рассмотрен также вариант быстрого градиентного метода Ю.Е. Нестерова для соответствующей концепции неточной модели оптимизируемого функционала. Для неускоренного метода не накапливаются все рассматриваемые типы погрешностей. Для ускоренного метода обоснована возможность уменьшения влияния переменной погрешности на оценку качества решения до любой приемлемой величины при накоплении величин, соответствующих постоянным значениям величин погрешностей используемой концепции модели оптимизируемого функционала. При этом адаптивность метода может на практике улучшать качество найденного решения по сравнению с полученными теоретическими оценками. Однако в полученных оценках качества решения реализована адаптивная настройка не всех параметров неточной модели. Полная адаптивная настройка на величины погрешностей возможна для искусственных неточностей, связанных например с рассмотрением негладких задач.

Обоснована применимость неускоренных процедур для относительно гладких целевых функционалов. В таком случае полученную оценку скорости сходимости $O(\varepsilon^{-1})$ можно считать оптимальной даже при отсутствии погрешностей \cite{Dragomir_1911.08510}. Показано, как можно ввести аналогичную концепцию неточной модели для вариационных неравенств и седловых задач и обосновать оценку скорости сходимости для аналога экстраградиентного метода с адаптивной настройкой на величину детерминированного шума.

В последнем шестом разделе работы рассмотрены адаптивные алгоритмические схемы с переключениями для негладких задач выпуклой оптимизации с липшицевыми ограничениями, которые достаточно эффективно работают для некоторых задач с целевыми функционалами более низкого уровня гладкости. В частности, речь может идти о задачах с дифференцируемыми гёльдеровыми целевыми функционалами \cite{Ivanova_1911.07354v2} или с относительно липшицевыми целевыми функционалами \cite{Lu_1710.04718v3} и функционалами ограничений. Введённые  релаксации условия Липшица, в частности, позволяют получать оценки скорости сходимости с относительной точностью для однородных целевых функционалов при достаточно общих предположениях. Однако в отличие от результатов первых разделов работы по методам градиентного типа не удалось предложить методы с адаптивной настройкой на величины погрешностей. Тем не менее, в таких схемах возможно использование $\delta$-субградиентов вместо обычных субградиентов, а также возмущенных с точностью $\delta>0$ значений функционалов. В итоговых оценках при этом не происходит накопление величин, соответствующих $\delta$ \cite{Nurminsky}. Это означает в частности что в итоговых оценках скорости сходимости не накапливаются величины, соответствующие погрешностям, возникающим при решении вспомогательных задач на итерациях алгоритма \ref{alg5}. Указанная методика применима к задачам с любым количеством функционалов ограничений. В этом плане интересной задачей может быть сравнение разработанной методики с подходом к негладким задачам выпуклого программирования с одним функционалом ограничения, который основан на переходе к одномерной двойственной задаче. При этом нахождение подходящего значения двойственного множителя может выполняться методом дихотомии \cite{Stonyakin_MOTOR_2019} при допустимой погрешности значения производной, связанной с неточностью решения вспомогательных подзадач. Для задач с двумя ограничениями можно применять методику \cite{Pasechnjuk_2019} с критерием остановки, похожим на критерий остановки алгоритма 1 из \cite{Stonyakin_MOTOR_2019} и соответствующим подходящему значению возмущенного градиента двойственной задачи. Экспериментально показано, что такие подходы могут приводить к линейной скорости сходимости даже для негладких задач. Если же целевой функционал гладкий и сильно выпуклый, а функционалы ограничений гладкие и выпуклые, то близкую к линейной скорость сходимости можно обосновать \cite{Stonyakin_MOTOR_2019}.

Статья опубликована в \cite{Stonyakin_2020}. Настоящая версия текста содержит некоторые уточнения, а также исправления допущенных ранее опечаток.

\end{document}